\documentclass[11pt,reqno]{amsart}
\UseRawInputEncoding

\usepackage{latexsym,bm, amsfonts,amsthm,amsmath,mathrsfs,CJK}

\makeatletter

\newcommand{\Rmnum}[1]{\expandafter\@slowromancap\romannumeral #1@}
\makeatother

 \newtheorem{lem}{Lemma}[section]  \newtheorem{thm}{Theorem}[section]
\newtheorem{cor}{Corollary}[section] \newtheorem{defn}{Definition}[section]  
\numberwithin{equation}{section}

 \newcommand{\me}{\mathrm{e}} 
\newcommand{\dif}{\mathrm{d}} \DeclareMathAlphabet{\mathsfsl}{OT1}{cmss}{m}{sl} \DeclareMathAlphabet{\mathpzc}{OT1}{pzc}{m}{it}

    \newcommand{\ee}{\mathbb{E}}

\newcommand{\nn}{\mathbb{N}}
\newcommand{\rr}{\mathbb{R}}
    
\newcommand{\vv}{\mathbb{V}}
 \def\CC{\mathcal C}   \def\FF{\mathcal F}  \def\HH{\mathcal H}

\def\d"{^{\prime\prime}} \def\d'{^{\prime}}

\begin{document}
\begin{CJK*}{GBK}{kai}
\title[]{Convergence rate in the law of logarithm for negatively dependent random variables under sub-linear expectations}\thanks{Project supported by Science and Technology Research Project of Jiangxi Provincial Department of Education of China (No. GJJ2201041), Doctoral Scientific Research Starting Foundation of Jingdezhen Ceramic University ( No.102/01003002031) , Academic Achievement Re-cultivation Project	 of Jingdezhen Ceramic University (Grant No. 215/20506277).}
\date{} \maketitle



\begin{center}
 XU Mingzhou~\footnote{Email: mingzhouxu@whu.edu.cn}\quad  WANG Wei\footnote{Email: 2093754218@qq.com}
 \\
 School of Information Engineering, Jingdezhen Ceramic University\\
  Jingdezhen 333403, China
\end{center}

 \renewcommand{\abstractname}{~}

{\bf Abstract:}
Let $\{X,X_n,n\ge 1\}$ be a sequence of identically distributed, negatively dependent (NA) random variables under sub-linear expectations, and denote $S_n=\sum_{i=1}^{n}X_i$, $n\ge 1$. Assume that $h(\cdot)$ is a positive non-decreasing function on $(0,\infty)$ fulfulling $\int_{1}^{\infty}(th(t))^{-1}\dif t=\infty$. Write $Lt=\ln \max\{\me,t\}$, $\psi(t)=\int_{1}^{t}(sh(s))^{-1}\dif s$, $t\ge 1$. In this sequel, we establish that $\sum_{n=1}^{\infty}(nh(n))^{-1}\vv\left\{|S_n|\ge (1+\varepsilon)\sigma\sqrt{2nL\psi(n)}\right\}<\infty$, $\forall \varepsilon>0$ if $\ee(X)=\ee(-X)=0$ and $\ee(X^2)=\sigma^2\in (0,\infty)$. The result generalizes that of NA random variables in probability space.	


\vspace{-3mm}

\section{Introduction }
Peng \cite{peng2007g,peng2019nonlinear} introduced  the concepts of the sub-linear expectations space to study the uncertainty in probability.  The works of Peng \cite{peng2007g,peng2019nonlinear} stimulated many scholars to investigate the results under sub-linear expectations space, which extend the relevant ones in probability space. Zhang \cite{zhang2016exponential, zhang2016rosenthal} got exponential inequalities and Rosenthal's inequality under sub-linear expectations. Under sub-linear expectations, Xu and Cheng \cite{xu2022small} obtained how small the increments of $G$-Brownian motion are, Xu and Kong \cite{xu2024completeqth} studied complete complete $q$th moment convergence of moving average processes for $m$-widely acceptable random variables. For more limit theorems under sub-linear expectations, the interested readers could refer to, Zhang \cite{zhang2015donsker}, Xu and Zhang \cite{xu2019three, xu2020law}, Wu and Jiang \cite{wu2018strong}, Zhong and Wu \cite{zhong2017complete}, Chen \cite{chen2016strong},  Zhang \cite{zhang2022strong}, Gao and Xu \cite{gao2011large}, Kuczmaszewska \cite{kuczmaszewska2020complete}, Xu and Cheng \cite{xu2021convergence, xu2022note,xu2022noteb}, Xu et al. \cite{xu2022complete}, \cite{xu2023convergence},  Xu and Kong \cite{xu2023noteon}, Chen and Wu \cite{chen2022complete}, Xu \cite{xu2023acomplete, xu2023completeconvergencea,  xu2024onthe, xu2024sample} and the references therein.
	
	In probability space, Qiu and Zhao \cite{qiu2023convergence} studied convergence rate in the law of logarithm for negatively dependent random variables. For references on complete rate in linear expectation space, the interested reader could refer to Li \cite{li2009convergencerate}, Zhang and Wu \cite{zhang2016moment}, and the refercences therein. Motivated by the works of  Qiu and Zhao \cite{qiu2023convergence}, we try to investigate convergence rate in the law of logarithm for negatively dependent random variables under sub-linear expectations, which extends the relevant ones in Qiu and Zhao \cite{qiu2023convergence}.
	
	The rest of this paper is constructed as follows. We present necessary basic notions, concepts, relevant  properties, and lemmas under sub-linear expectations in the next section. In Section 3, we give our main result, Theorem \ref{thm01},  the proof of which is presented in Section 4.

	\section{Preliminary}
	\setcounter{equation}{0}
	\setcounter{equation}{0}
	As in Xu and Cheng \cite{xu2021convergence}, we adopt similar notations as in the work by Peng\cite{peng2019nonlinear}, Zhang \cite{zhang2016exponential}. Assume that $(\Omega,\FF)$ is a given measurable space. Suppose that $\HH$ is a subset of all random variables on $(\Omega,\FF)$ such that $X_1,\cdots,X_n\in \HH$ yields $\varphi(X_1,\cdots,X_n)\in \HH$ for each $\varphi\in \CC_{l,Lip}(\rr^n)$, where $\CC_{l,Lip}(\rr^n)$ means the linear space of (local lipschitz) function $\varphi$ fulfilling
	$$
	|\varphi(\mathbf{x})-\varphi(\mathbf{y})|\le C(1+|\mathbf{x}|^m+|\mathbf{y}|^m)(|\mathbf{x}-\mathbf{y}|), \forall \mathbf{x},\mathbf{y}\in \rr^n
	$$
	for some $C>0$, $m\in \nn$ relying on $\varphi$.
	\begin{defn}\label{defn1} A sub-linear expectation $\ee$ on $\HH$ is a functional $\ee:\HH\mapsto \bar{\rr}:=[-\infty,\infty]$ fulfilling the following properties: for all $X,Y\in \HH$, we have
		\begin{description}
			\item[\rm (a)] If $X\ge Y$, then $\ee[X]\ge \ee[Y]$;
			\item[\rm (b)] $\ee[c]=c$, $\forall c\in\rr$;
			\item[\rm (c)] $\ee[\lambda X]=\lambda\ee[X]$, $\forall \lambda\ge 0$;
			\item[\rm (d)] $\ee[X+Y]\le \ee[X]+\ee[Y]$ whenever $\ee[X]+\ee[Y]$ is not of the form $\infty-\infty$ or $-\infty+\infty$.
		\end{description}
		
	\end{defn}
	A set function $V:\FF\mapsto[0,1]$ is called to be a capacity if
	\begin{description}
		\item[\rm (a)]$V(\emptyset)=0$, $V(\Omega)=1$;
		\item[\rm (b)]$V(A)\le V(B)$, $A\subset B$, $A,B\in \FF$.\\
	\end{description}
	A capacity $V$ is said to be sub-additive if $V(A\bigcup B)\le V(A)+V(B)$, $A,B\in \FF$.

	In this article, given a sub-linear expectation space $(\Omega, \HH, \ee)$, define $\vv(A):=\inf\{\ee[\xi]:I_A\le \xi, \xi\in \HH\}$, $\forall A\in \FF$ (see (2.3) and the definitions of $\vv$ above (2.3) in Zhang \cite{zhang2016exponential}). $\vv$ is a sub-additive capacity. Write
	$$
	C_{\vv}(X):=\int_{0}^{\infty}\vv(X>x)\dif x +\int_{-\infty}^{0}(\vv(X>x)-1)\dif x.
	$$
	
	Assume that $\mathbf{X}=(X_1,\cdots, X_m)$, $X_i\in\HH$ and $\mathbf{Y}=(Y_1,\cdots,Y_n)$, $Y_i\in \HH$  are two random vectors on  $(\Omega, \HH, \ee)$. $\mathbf{Y}$ is named to be negatively dependent to $\mathbf{X}$, if for each $\psi_1\in \CC_{l,Lip}(\rr^m)$, $\psi_2\in \CC_{l,Lip}(\rr^n)$, we have $\ee[\psi_1(\mathbf{X})\psi_2(\mathbf{Y})]\le\ee[\psi_1(\mathbf{X})] \ee[\psi_2(\mathbf{Y})]$ whenever $\psi_1(\mathbf{X})\ge 0$, $\ee[\psi_2(\mathbf{Y})]\ge 0 $, $\ee[|\psi_1(\mathbf{X})\psi_2(\mathbf{Y})|]<\infty$, $\ee[|\psi_1(\mathbf{X})|]<\infty$, $\ee[|\psi_2(\mathbf{Y})|]<\infty$, and either $\psi_1$ and $\psi_2$ are coordinatewise nondecreasing or $\psi_1$ and $\psi_2$ are coordinatewise nonincreasing (see Definition 2.3 of Zhang \cite{zhang2016exponential}, Definition 1.5 of Zhang \cite{zhang2016rosenthal}).
	$\{X_n\}_{n=1}^{\infty}$ is called to be a sequence of negatively dependent random variables, if $X_{n+1}$ is negatively dependent to $(X_1,\cdots,X_n)$ for each $n\ge 1$.
	
	Suppose that $\mathbf{X}_1$ and $\mathbf{X}_2$ are two $n$-dimensional random vectors defined, respectively, in sub-linear expectation spaces $(\Omega_1,\HH_1,\ee_1)$ and $(\Omega_2,\HH_2,\ee_2)$. They are said to be identically distributed if  for every $\psi\in \CC_{l,Lip}(\rr^n)$ such that $\psi(\mathbf{X}_1)\in \HH_1, \psi(\mathbf{X}_2)\in \HH_2$,
	$$
	\ee_1[\psi(\mathbf{X}_1)]=\ee_2[\psi(\mathbf{X}_2)], \mbox{  }
	$$
	whenever the sub-linear expectations are finite. $\{X_n\}_{n=1}^{\infty}$ is called to be identically distributed if for each $i\ge 1$, $X_i$ and $X_1$ are identically distributed.
	
	We cite two lemmas below.
	\begin{lem}\label{lem01}(cf. Theorem 3.1 (a) of Zhang \cite{zhang2016exponential})
		Suppose that $\{X_n,n\ge 1\}$ is a sequence of NA random variables with $\ee(X_n)\le 0$ in sub-linear expectation space $(\Omega, \HH, \ee)$. Write $B_n=\sum_{i=1}^{n}\ee(X_i^2)$. Then for any $x>0$, $\alpha>0$, we have
		\begin{eqnarray*}
			\vv\left\{S_n\ge x\right\}\le \vv\left\{\max_{1\le j\le n}X_j> \alpha\right\}+\exp\left\{-\frac{x^2}{2(\alpha x+B_n)}\left[1+\frac{2}{3}\ln\left(1+\frac{\alpha x}{B_n}\right)\right]\right\}.
		\end{eqnarray*}
	\end{lem}
	\begin{lem}\label{lem02}
		Assume that that $\{X_n,n\ge 1\}$ is a sequence of NA random variables with $\ee(X_n)\le 0$ in sub-linear expectation space $(\Omega, \HH, \ee)$. Then for $1\le \nu\le 2$, we have
		\begin{eqnarray*}
			\ee\left\{\left(\left(\sum_{i=1}^{n}X_i\right)^{+}\right)^{\nu}\right\}\le 4\sum_{k=1}^{n}\ee|X_k|^{\nu}.
		\end{eqnarray*}
	\end{lem}
	\begin{proof}By Theorem 2.1 (a) of Zhang \cite{zhang2016rosenthal}, we immediately see that the result here holds.
	\end{proof}
	
	In the paper we assume that $\ee$ is countably sub-additive, i.e., $\ee(X)\le \sum_{n=1}^{\infty}\ee(X_n)$, whenever $X\le \sum_{n=1}^{\infty}X_n$, $X,X_n\in \HH$, and $X\ge 0$, $X_n\ge 0$, $n=1,2,\ldots$. Let $C$ stands for a positive constant which may differ from place to place. $I(A)$ or $I_A$ represents the indicator function of $A$.

	\section{Main Results}
	The following is our main result.
	\begin{thm}\label{thm01}
		Assume that $\{X,X_n, n\ge 1\}$ is a sequence of identically distributed NA random variables in sub-linear expectation space $(\Omega, \HH, \ee)$, $h(\cdot)$ is a positive non-decreasing function on $(0,\infty)$ fulfilling  $\int_{1}^{\infty}(th(t))^{-1}\dif t=\infty$. Write $\psi(t)=\int_{1}^{t}(sh(s))^{-1}\dif s$, $t\ge 0$. If
		\begin{eqnarray}\label{01}
			&\ee(X)=\ee(-X)=0,  \ee(X^2)=\bar{\sigma}^2\le C_{\vv}\left(X^2\right)<\infty,
		\end{eqnarray}
		then  for any $\varepsilon>0$, we have
		\begin{eqnarray}\label{02}
			\sum_{n=1}^{\infty}\frac{1}{nh(n)}\vv\left\{S_n\ge (1+\varepsilon)\bar{\sigma}\sqrt{2nL\psi(n)}\right\}<\infty,
		\end{eqnarray}
		\begin{eqnarray}\label{02+}
			\sum_{n=1}^{\infty}\frac{1}{nh(n)}\vv\left\{-S_n\ge (1+\varepsilon)\bar{\sigma}\sqrt{2nL\psi(n)}\right\}<\infty,
		\end{eqnarray}
	\end{thm}
	
	By the similar discussion of Li and Rosalsky \cite{li2007assuplement}, by Theorem \ref{thm01}, we could obtain the following.
	\begin{cor}\label{cor1} Suppose that $b\ge 0$, $\{X,X_n, n\ge 1\}$ is a sequence of identically distributed NA random variables in sub-linear expectation space $(\Omega, \HH, \ee)$. If (\ref{01}) holds, then we have
		\begin{eqnarray*}
			&&	\sum_{n=1}^{\infty}\frac{1}{n(LLn)^b}\vv\left(S_n\ge (1+\varepsilon)\bar{\sigma}\sqrt{2nLLn}
			\right)<\infty, \quad \forall \varepsilon>0,\\
			&&\sum_{n=1}^{\infty}\frac{1}{n(LLn)^b}\vv\left(-S_n\ge (1+\varepsilon)\bar{\sigma}\sqrt{2nLLn}
			\right)<\infty, \quad \forall \varepsilon>0.
		\end{eqnarray*}
	\end{cor}
	
	\begin{cor}\label{cor2}
		Assume that $0\le r<1$, $\{X,X_n, n\ge 1\}$ is a sequence of identically distributed NA random variables in sub-linear expectation space $(\Omega, \HH, \ee)$. If (\ref{01}) holds, then
		\begin{eqnarray*}
			&&	\sum_{n=1}^{\infty}\frac{1}{n(Ln)^r}\vv\left(|S_n|\ge (1+\varepsilon)\bar{\sigma}\sqrt{2(1-r)nLLn}
			\right)<\infty, \quad \forall \varepsilon>0,\\
			&&	\sum_{n=1}^{\infty}\frac{1}{n Ln}\vv\left(|S_n|\ge (1+\varepsilon)\bar{\sigma}\sqrt{2nLLLn}
			\right)<\infty, \quad \forall \varepsilon>0.
		\end{eqnarray*}
	\end{cor}
	\section{Proof of major result}
	\emph{Proof of Theorem \ref{thm01}}
We only prove (\ref{01}) $\Rightarrow$ (\ref{02}), since (\ref{01}) $\Rightarrow$ (\ref{02+}) can be proved similarly. Without loss of generalities, we suppose that $\ee(X^2)=1$. Then for any $v>0$, write
\[
f(t)=(1+v)^2\left[1+\frac23\ln(1+t)\right]-(1+t), \quad t>-1.
\]
Then $f(t)$ is continuous on $(-1,\infty)$ and $f(0)>0$. Then we see that there exists $t_0(v)>0$ such that $f(t)>0$, $t\in [0,t_0(v)]$. For any $\varepsilon>0$, we choose $\varepsilon_i>0, i=1,2, 3$ such that $\sum_{i=1}^{3}\varepsilon_i=\varepsilon$. For $v=\varepsilon_1$, by the discussions as mentioned above, $\exists t_0(\varepsilon_1)>0$ such that $f(t_0(\varepsilon_1))>0$, and
\[
\frac{(1+\varepsilon_1)^2}{t_0(\varepsilon_1)+1}\left[1+\frac{2}{3}\ln(1+t_0(\varepsilon_1))\right]>1.
\]
Set
\[
\theta(\varepsilon_1)=-\frac{(1+\varepsilon_1)^2}{t_0(\varepsilon_1)+1}\left[1+\frac{2}{3}\ln(1+t_0(\varepsilon_1))\right].
\]
Then
\begin{eqnarray}\label{03}
	\theta(\varepsilon_1)<-1.
\end{eqnarray}
Write $\delta_n=\sqrt{n/[2L\psi(n)]}$, $n\ge 1$. For $1\le i \le n$, $n \ge 1$, denote
\begin{eqnarray*}
	X_{ni}^{(1)}=X_iI\left(X_i\le \frac{t_0(\varepsilon_1)}{1+\varepsilon_1}\delta_n\right) +\frac{t_0(\varepsilon_1)}{1+\varepsilon_1}\delta_n I\left(X_i>\frac{t_0(\varepsilon_1)}{1+\varepsilon_1}\delta_n\right),
\end{eqnarray*}
\begin{eqnarray*}
	X_{ni}^{(2)}=\left(X_i-\frac{t_0(\varepsilon_1)}{1+\varepsilon_1}\delta_n\right)I\left(\frac{t_0(\varepsilon_1)}{1+\varepsilon_1}\delta_n<X_i\le \sqrt{2n}\right)+\left(\sqrt{2n}-\frac{t_0(\varepsilon_1)}{1+\varepsilon_1}\delta_n\right)I\left(X_i>\sqrt{2n}\right),
\end{eqnarray*}
\begin{eqnarray*}
	X_{ni}^{(3)}=\left(X_i-\sqrt{2n}\right)I(X_i>\sqrt{2n}),
\end{eqnarray*}
and $X^{(k)}$ is defined similarly as $X_{ni}^{(k)}$ only with $X$ in place of $X_i$ above. Obviously $S_{nk}^{(j)}=\sum_{i=1}^{k}X_{ni}^{(j)}$, $j=1,2,3$, $1\le k\le n$, $n\ge 1$.  Then we see that
\begin{eqnarray*}
	&&\quad	\sum_{n=1}^{\infty}\frac{1}{nh(n)}\vv\left\{S_n\ge (1+\varepsilon)\sqrt{2nL\psi(n)}\right\}\\
	&&\le \sum_{n=1}^{\infty}\frac{1}{nh(n)}\vv\left\{S_{nn}^{(1)}\ge (1+\varepsilon_1)\sqrt{2nL\psi(n)}\right\}\\
	&&\quad +\sum_{j=2}^{3}\sum_{n=1}^{\infty}\frac{1}{nh(n)}\vv\left\{S_{nn}^{(j)}\ge \varepsilon_j\sqrt{2nL\psi(n)}\right\}\\
	&&=:\Rmnum{1}_1+\Rmnum{1}_2+\Rmnum{1}_3.
\end{eqnarray*} By the definition of $\psi(t)$, $0<\psi(t)\le h(1)^{-1}Lt$, $t>1$. We see that
\begin{eqnarray}\label{04}
	\lim_{n\rightarrow\infty}\delta_n=\infty.
\end{eqnarray}
Now we first prove $\Rmnum{1}_1<\infty$.
By the property of NA random variables, $\{X_{ni}^{(1)}, 1\le i\le n\}$ is also negatively dependent. For $\{X_{ni}^{(1)}, 1\le i\le n\}$, by Lemma \ref{lem01} and $\ee(X_{ni}^{(1)})\le \ee(X_i)= 0$, taking $x=(1+\varepsilon_1)\sqrt{2nL\psi(n)}$, $\alpha=t_0(\varepsilon_1)(1+\varepsilon_1)^{-1}\delta_n$, observing that $\max_{1\le i\le n}X_{ni}^{(1)}\le t_0(\varepsilon_1)(1+\varepsilon_1)^{-1}\delta_n=\alpha$, $\sum_{i=1}^{n}\ee\left[(X_{ni}^{(1)})^2\right]\le n$, we see that
\begin{eqnarray*}
	&&	\vv\left\{S_{nn}^{(1)}\ge (1+\varepsilon_1)\sqrt{2nL\psi(n)}\right\}\\
	&&\quad\le \exp\left\{-\frac{(1+\varepsilon_1)^2L\psi(n)}{t_0(\varepsilon_1)+1}\left[1+\frac23\ln(1+t_0(\varepsilon_1))\right]\right\}\\
	&&\quad \le \exp\left\{\theta(\varepsilon_1)L\psi(n)\right\}=\psi(n)^{\theta(\varepsilon_1)}.
\end{eqnarray*}
Hence, by (\ref{03}), we conclude that
\begin{eqnarray*}
	\Rmnum{1}_1\le C\sum_{n=1}^{\infty}\frac{1}{nh(n)}\psi(n)^{\theta(\varepsilon_1)}\le C\int_{1}^{\infty}\frac{1}{xh(x)\psi(x)^{-\theta(\varepsilon_1)}}\dif x\le C\int_{\psi(1)}^{\infty}t^{\theta(\varepsilon_1)}\dif t<\infty.
\end{eqnarray*}

Next we prove $\Rmnum{1}_2<\infty$. By $\ee(X^2)=1$, the definition of $X_{nk}^{(2)}$ and (\ref{04}), we obtain
\begin{eqnarray*}
	0&\le&\frac{1}{\sqrt{2nL\psi(n)}}\sum_{k=1}^{n}\ee X_{nk}^{(2)}=\frac{n}{\sqrt{2nL\psi(n)}}\ee X^{(2)}\\
	&\le&\frac{n}{\sqrt{2nL\psi(n)}}C_{\vv}\left\{|X|I\left(|X|>\frac{t_0(\varepsilon_1)}{1+\varepsilon_1}\delta_n\right)\right\}\\
	 &\le&C\left[\left(\frac{t_0(\varepsilon_1)}{1+\varepsilon_1}\delta_n\right)^2\vv\left\{|X|>\frac{t_0(\varepsilon_1)}{1+\varepsilon_1}\delta_n\right\}+\int_{\frac{t_0(\varepsilon_1)}{1+\varepsilon_1}\delta_n}^{\infty}y\vv\left\{|X|>y\right\}\dif y\right]\\
	&\rightarrow&0 \text{   as $n\rightarrow\infty$.}
\end{eqnarray*}
By the definition of $X_{nk}^{(2)}$ and NA random variables, we know that $\{X_{(ni)}^{(2)},1\le i\le n\}$ is negatively dependent under sub-linear expectations. By Lemma \ref{lem02}, Markov's inequality under sub-linear expectations, $C_r$ inequality, Jensen's inequality under sub-linear expectations (cf. Lin \cite{lin2019jensen}), Lemma 2.3 of Xu \cite{xu2024onthe},  when $1<\tau<2$, we have
\begin{eqnarray*}
	\Rmnum{1}_{2}&\le&C\sum_{n=1}^{\infty}\frac{1}{nh(n)}\vv\left\{\sum_{i=1}^{n}X_{ni}^{(2)}\ge \varepsilon_2\sqrt{2nL\psi(n)}\right\}\\
	&\le&C+C\sum_{n=1}^{\infty}\frac{1}{nh(n)}\vv\left\{\sum_{i=1}^{n}\left(X_{ni}^{(2)}-\ee X_{ni}^{(2)}\right)\ge \varepsilon_2/2\sqrt{2nL\psi(n)}\right\}\\
	&\le&C+C\sum_{n=1}^{\infty}\frac{1}{nh(n)}\frac{1}{(\sqrt{2nL\psi(n)})^{\tau}}\ee\left(\left(\sum_{i=1}^{n}\left[X_{ni}^{(2)}-\ee X_{ni}^{(2)}\right]\right)^{+}\right)^{\tau}\\
	&\le&C+C\sum_{n=1}^{\infty}\frac{1}{nh(n)}\frac{1}{(\sqrt{2nL\psi(n)})^{\tau}}\sum_{i=1}^{n}\ee|X_{ni}^{(2)}|^{\tau}\\
	&\le&C+C\sum_{n=1}^{\infty}\frac{1}{nh(n)}\frac{1}{(\sqrt{2nL\psi(n)})^{\tau}}n\ee|X^{(2)}|^{\tau}\\
	&\le&C+C\sum_{n=1}^{\infty}\frac{1}{h(n)}\frac{1}{(\sqrt{2nL\psi(n)})^{\tau}}C_{\vv}\left\{|X^{(2)}|^{\tau}\right\}\\
	 &\le&C+C\sum_{n=1}^{\infty}\frac{1}{h(n)}\frac{1}{(\sqrt{2nL\psi(n)})^{\tau}\delta_n^{\tau}}C_{\vv}\left\{X^{2\tau}I\left\{\frac{t_0(\varepsilon_1)}{1+\varepsilon_1}\delta_n <X\le \sqrt{2n}\right\}+(2n)^{\tau}I\{X>\sqrt{2n}\}\right\}\\
	&\le&C+C\sum_{n=1}^{\infty}\frac{1}{n^{\tau}}C_{\vv}\left\{|X|^{2\tau}I\left\{|X|\le \sqrt{2n}\right\}\right\}+C\sum_{n=1}^{\infty}\vv\left\{X>\sqrt{2n}\right\}\\
	&\le&C+C\sum_{n=1}^{\infty}\frac{1}{n^{\tau}}\int_{0}^{\sqrt{2n}}\vv\left\{|X|>y\right\}y^{2\tau-1}\dif y+CC_{\vv}\left\{|X|^2\right\}\\
	&\le&C+C\int_{1}^{\infty}\frac{1}{x^{\tau}}\dif x\int_{0}^{\sqrt{2x}}\vv\left\{|X|>y\right\}y^{2\tau-1}\dif y+CC_{\vv}\left\{|X|^2\right\}\\
	&\le&C+\int_{0}^{\infty}\vv\left\{|X|>y\right\}y^{2\tau-1}\dif y\int_{\frac{y^2}{2}\bigvee 1}^{\infty}\frac{1}{x^{\tau}}\dif x+CC_{\vv}\left\{|X|^2\right\}\\
	&\le&C+C\int_{0}^{\infty}\vv\left\{|X|>y\right\}\left(\frac{y^2}{2}\bigvee 1\right)^{-\tau+1}y^{2\tau-1}\dif y+CC_{\vv}\left\{|X|^2\right\}\\
	&\le&C+CC_{\vv}\left\{|X|^2\right\}<\infty.
\end{eqnarray*}
Finally, we prove $\Rmnum{1}_3<\infty$. By the discussion above (2.8) of Zhang \cite{zhang2021onthelaw}, we see that
\begin{eqnarray*}
	\Rmnum{1}_3&\le&C\sum_{n=1}^{\infty}\frac{1}{nh(n)}\vv\left\{\bigcup_{i=1}^n\left\{|X_i|>\sqrt{2n}\right\}\right\}\\
	&\le&C\sum_{n=1}^{\infty}\vv\left\{|X|>\sqrt{2n}\right\}\le CC_{\vv}\left\{|X|^2\right\}<\infty.
\end{eqnarray*}
This completes the proof. ~$\Box$









\end{CJK*}
\end{document}